\documentclass[11pt,reqno]{amsart}
\usepackage{amsmath, amssymb, amsthm}
\usepackage{url}
\usepackage[breaklinks]{hyperref}
\usepackage{array}
\renewcommand{\(}{\left\(}
\renewcommand{\)}{\right\)}
\renewcommand{\[}{\left\[}
\renewcommand{\]}{\right\]}

\numberwithin{equation}{section}
\theoremstyle{plain}
\newtheorem{theorem}{Theorem}[section]

\newtheorem{remark}[]{Remark}

\newtheorem{conjecture}[theorem]{Conjecture}
\theoremstyle{definition}

\theoremstyle{remark}

\numberwithin{equation}{section}

\begin{document}
	\title[Discussion on certain conjectures regarding the periodicity of sign patterns]{Discussion on some conjectures regarding the periodicity of sign patterns of certain infinite products involving the Rogers-Ramanujan Continued Fractions} 
	
	\author{Suparno Ghoshal and Arijit Jana}
		\address{Department of Computer Science, Ruhr University Bochum, Germany}
	\email{suparno.ghoshal@rub.de}
	\address{Department of Mathematics, National Institute of Technology, Silchar, 788010, India}
	\email{jana94arijit@gmail.com}
	
	\thanks{2020 \textit{Mathematics Subject Classification.} 11P81, 11P83, 05A17.\\
		\textit{Keywords and phrases. Rogers-Ramanujan functions, periodicity}}
	
\begin{abstract}
	Let $R(q)$ denote the Rogers-Ramanujan continued fraction. Define
$$ \frac{1}{R^5(q)}=\displaystyle \sum_{n=0}^{\infty}A(n)q^{n} \quad  \text{and}   \quad R^5(q)=\displaystyle\sum_{n=0}^{\infty}B(n)q^{n}.$$
   Baruah and Sarma recently posed conjectures regarding the sign patterns of $A(5n), B(5n)$ for $n\geq 0.$  In this paper, we show that these conjectures do not hold for $n=0$.
\end{abstract}
	\maketitle
\section{Introduction and statement of results}
  For complex number $a$ and $q$,   the $q$- rising factorial is defined  by
\begin{align*}
(a;q)_0:=1,~&(a;q)_n:=\prod_{k=0}^{n-1}(1-aq^k) ~\textup{for}~ n\geq1, \\
&(a;q)_\infty:=\prod_{k=0}^{\infty}(1-aq^k) ~\textup{for}~ |q|<1.
\end{align*}
Throughout the paper, we set $f_k:=(q^k;q^k)_\infty$. 
The renowned Rogers-Ramanujan continued fraction is defined by 
$$\mathcal{R}(q):=\dfrac{q^{1/5}}{1}_{+}\dfrac{q}{1}_{+}\dfrac{q^2}{1}_{+}\dfrac{q^3}{1}_{+~\cdots},\quad|q|<1.$$
Set $$R(q):=q^{-1/5}\mathcal{R}(q):=\dfrac{1}{1}_{+}\dfrac{q}{1}_{+}\dfrac{q^2}{1}_{+}\dfrac{q^3}{1}_{+~\cdots}.$$

The Rogers-Ramanujan identities are given by
\begin{align*}
G(q)=\sum_{n=0}^{\infty}\frac{q^{n^2}}{(q;q)_n}=\frac{1}{(q;q^5)_{\infty}(q^4;q^5)_{\infty}}\\\intertext{and}
H(q)=\sum_{n=0}^{\infty}\frac{q^{n^2+n}}{(q;q)_n}=\frac{1}{(q^2;q^5)_{\infty}(q^3;q^5)_{\infty}},
\end{align*}
In $1894$, Rogers \cite{rogers1894} proved that  
\begin{align}\label{R-GH}
R(q)&=\dfrac{H(q)}{G(q)}=\dfrac{(q;q^5)_\infty(q^4;q^5)_\infty}{(q^2;q^5)_\infty(q^3;q^5)_\infty}.
\end{align}
The above identity was also given by  Ramanujan \cite{nb} (see \cite[Corollary, p. 30]{bcb3}). 
The power series coefficients of the infinite products given in \eqref{R-GH} and their reciprocal were asymptotically studied by Richard and Szekeres \cite[Eq. (3.9)]{richmond1978taylor} in $1978$.
 In particular, they  proved that, if 
 \begin{align}\label{eq1}
 R(q) =:\sum_{n=0}^{\infty} d(n)q^n =\dfrac{(q;q^5)_\infty(q^4;q^5)_\infty}{(q^2;q^5)_\infty(q^3;q^5)_\infty},
 \end{align}
 then for $n$ sufficiently large,
 \begin{align}\label{thm:rich-2} 
 &d(5n)>0,~d(5n+1)<0,~ d(5n+2)>0,~d(5n+3)<0,~ \text{and}~d(5n+4)<0. 
 \end{align}
 Also, if
\begin{align}\label{eq2}
\frac{1}{R(q)} =:\sum_{n=0}^{\infty} c(n)q^n =\dfrac{(q^2;q^5)_\infty(q^3;q^5)_\infty}{(q;q^5)_\infty(q^4;q^5)_\infty},	
\end{align}
then
\begin{align*}
c(n)=\dfrac{\sqrt{2}}{(5n)^{3/4}}\textup{exp}\left(\dfrac{4\pi}{25}\sqrt{5n}\right)\times\left\{\cos\left(\dfrac{2\pi}{5}\left(n-\dfrac{2}{5}\right)\right)+\mathcal{O}(n^{-1/2})\right\},
\end{align*}

which implies that, for $n$ sufficiently large,
\begin{align}\label{thm:rich}
c(5n)>0, ~c(5n+1)>0, ~c(5n+2)<0,~c(5n+3)<0,~\text{and}~ c(5n+4)<0.   
\end{align}

In his lost notebook \cite[p. 50]{lnb}, Ramanujan recorded formulas for the generating functions
\begin{align*}
\sum_{n=0}^\infty c(5n+j)q^n \quad \text{and} \quad \sum_{n=0}^\infty d(5n+j)q^n, \quad \text{for}\quad 0\leq j\leq 4
\end{align*}
These identities were later proved by Andrews \cite{andrews1981ramunujan}. Using these formulas along with a theorem of Gordon \cite{gordon}, Andrews provided partition-theoretic interpretations for the coefficients $c(n), d(n)$ and consequently established that \eqref{thm:rich} and \eqref{thm:rich-2} hold for all $n$, except in the cases
$c(2)=c(4)=c(9)=0$, $d(3)=d(8)=0$. Building on this work, Hirschhorn \cite{hirschhorn} applied the quintuple product identity \cite{cooper-qtpi} to derive exact  $q$-product representations for the same generating functions. His analysis revealed a periodic pattern in the signs of the coefficients
$c(n)$ and $d(n)$, with two additional exceptions: $d(13)=d(23)=0$.

Ramanujan noted the following sophisticated identity in his second notebook \cite[p. 289]{nb} and the lost notebook \cite[p. 365]{lnb}. 
\begin{align}\label{R_1^5/R_5}	R^5(q)=R(q^5)\cdot\frac{1-2qR(q^5)+4q^2R^2(q^5)-3q^3R^3(q^5)+q^4R^4(q^5)}{1+3qR(q^5)+4q^2R^2(q^5)+2q^3R^3(q^5)+q^4R^4(q^5)}.
\end{align}
His first letter to Hardy, dated January $16, 1913$, also contains the identity.
  For the proofs of the above identity, one can see the papers by Rogers \cite{rogers1921}, Watson \cite{watson}, Ramanathan \cite{ramanathan-acta}, Yi \cite{yi}, and Gugg \cite{gugg-rama}.

Recently, Baruah and Sarmah \cite{baruah-sarma} investigate the behavior of the signs of the coefficients of the infinite products  $R^5(q)$, $\frac{1}{R^5(q)}$, $\frac{R^5(q)}{R(q^5)}$, and $\frac{R(q^5)}{R^5(q)}$. They mainly give the following results.
\begin{theorem}\cite[Theorem 2]{baruah-sarma}\label{periodR}
	 If $A(n)$ is defined by 
	\begin{align*}
	\dfrac{1}{R^5(q)}&=\sum_{n=0}^{\infty}A(n)q^{n},
	\end{align*}then for all nonnegative integers $n$, we have
	\begin{align*}
	A(5n+1)>0, 
	A(5n+2)>0,
	A(5n+3)>0,
	A(5n+4)<0.
	\end{align*}
\end{theorem}

\begin{theorem}\cite[Theorem 3]{baruah-sarma}\label{periodalpha}
	 If $B(n)$ is defined by 
	\begin{align*}
	R^5(q)&=\sum_{n=0}^{\infty}B(n)q^{n},
	\end{align*}then for all nonnegative integers $n$, we have
	\begin{align*}
	B(5n+1)<0,
	B(5n+2)>0,
	B(5n+3)<0,
	B(5n+4)>0.
	\end{align*}
\end{theorem}
\begin{theorem} \cite[Theorem 4]{baruah-sarma}\label{periodbeta}
	If $C(n)$ is defined by 
	\begin{align*}
	\dfrac{R^5(q)}{R(q^5)}&=\sum_{n=0}^{\infty}C(n)q^{n},
	\end{align*}then for all nonnegative integers $n$, we have
	\begin{align*}
	C(5n)<0, 
	C(5n+1)<0, 
	C(5n+2)>0, 
	C(5n+3)<0, 
	C(5n+4)>0,
	\end{align*}
	except $C(0)=1$.
\end{theorem}

\begin{theorem}\cite[Theorem 5]{baruah-sarma}\label{periodgamma}
	If $D(n)$ is defined by 
	\begin{align*}
	\dfrac{R(q^5)}{R^5(q)}&=\sum_{n=0}^{\infty}D(n)q^{n},
	\end{align*}then for all nonnegative integers $n$, we have
	\begin{align*}
	D(5n)<0, 
	D(5n+2)>0, 
	D(5n+3)>0, 
	D(5n+4)<0, 
	\end{align*}
	except $D(0)=1$.
\end{theorem}
In Theorems \ref{periodR}, \ref{periodalpha}, and \ref{periodgamma}, Baruah and Sarma gave the sign patterns of the coefficients $A(n)$, $B(n)$, and $D(n)$ of $1/R^5(q)$, $R^5(q)$, and $R(q^5)/R^5(q)$, respectively, except $A(5n)$, $B(5n)$, and $D(5n+1)$.  In the same paper, they posed the following conjecture based on numerical observation.

\begin{conjecture} \cite[Conjecture 13]{baruah-sarma}\label{conj 1}
	For all integers $n\geq0$, 
	\begin{align*}
	A(5n)&<0, \\
	B (5n)&<0,\\
	D(5n+1)&>0.		
	\end{align*}
\end{conjecture}
Our main results are stated below.
\begin{theorem}\label{gj1}

    \begin{align*}
       & A(0) >0, A(10)<0, A(15) < 0. \
       \end{align*}
    
\end{theorem}
\begin{remark}
	Conjecture \ref{conj 1} fails for $A(5n)$ at $n=0$.
\end{remark}
\begin{theorem}\label{gj2}
	\begin{align*}
	B(0)>0,	B(5) <0.
	\end{align*}

\end{theorem}
\begin{remark}
	Conjecture \ref{conj 1} fails for $B(5n)$ at $n=0$.
\end{remark}
\begin{theorem}\label{gj3}
	\begin{align*}
D(1)>0.	
\end{align*}
	
\end{theorem}

\section{ Proof of the main results }

\begin{proof}[Proof of  Theorem \ref{gj1}]
	We know from \cite[Equation (62)]{baruah-sarma} the following: 
	\begin{align*}
	\sum_{n = 0}^{\infty} A(n) q^{n} = \frac{f^{6}_{25}}{f^{6}_{5}\cdot R^{5}(q^5)}&\biggl(1 + 3qR(q^5)+4q^2 R^{2}(q^5) + 2q^{3}R^{3}(q^{5}) + q^{4}R^{4}(q^5)\biggl)^{2}\\
	&\times \biggl(\frac{1}{R(q^5)} - q - q^{2}R(q^5)\biggl).
		\end{align*}
	
	Now by extracting the terms of the form $q^{5n}$ from the above equation and by replacing $q^{5}$ by $q$ and by applying the result
	\begin{align}\label{B20}
	\frac{1}{R^5(q)}-q^2R^5(q)&=11q+\frac{f_1^6}{f_5^6}
	\end{align}  from \cite[Theorem 7.4.4]{Spirit}, we get following 
	\begin{align}\label{eq3} \notag
	\sum_{n = 0}^{\infty} A(5n) q^{n}& = \frac{1}{R(q)}(1 - 25q \frac{f^{6}_{5}}{f^{6}_{1}})\\ 
	& =\sum_{n=0}^{\infty}c(n)q^{n} . \sum_{n=0}^{\infty}C(5n)q^{n}
	\end{align}
	The last equality holds because of \eqref{eq2} and the fact that $ \displaystyle\sum_{n=0}^{\infty}C(5n)q^{n}=1 - 25q \frac{f^{6}_{5}}{f^{6}_{1}}$.
	Now equating the constant term of the both sides, we get $$ A(0)=c(0)C(0).$$
	Since, $c(0)>0,$ and $C(0)=1$, we obtain $A(0)>0.$ Now equating the coefficients of $q^2$  from \eqref{eq3}, we get 
	\begin{align*}
	A(10)= c(0)C(10)+c(1)C(5)+c(2)C(0)
	\end{align*}
Therefore, $A(10)<0$ because of $c(0), c(1)$ are positive, $C(10), C(5)$ are negative and $c(2)=0.$
Similarly,  equating the coefficients of $q^3$  from \eqref{eq3}, we get  
\begin{align*}
A(15)= c(0)C(15)+c(1)C(10)+c(2)C(5)+c(3)C(0)
\end{align*}
Clearly, $A(15)<0$.
\end{proof}

\begin{proof}[Proof of  Theorem \ref{gj2}]
	In order to prove the Theorem \ref{gj2}, we need to first use  \cite[Equation (69)]{baruah-sarma} to arrive at the following: 
	\begin{align*} 
	\sum_{n = 0}^{\infty} B(n) q^{n} = \frac{f^{6}_{25}}{f^{6}_{5}\cdot R^{3}(q^5)}&\biggl(1 - 2qR(q^5)+4q^2 R^{2}(q^5) - 3q^{3}R^{3}(q^{5}) + q^{4}R^{4}(q^5)\biggl)^{2}\\
	&\times \biggl(\frac{1}{R(q^5)} - q - q^{2}R(q^5)\biggl).
	\end{align*}
	Now by extracting the terms of the form $q^{5n}$ from the above equation and by replacing $q^{5}$ by $q$ and by applying \eqref{B20}, we get following:
	\begin{align}\label{eq4}
\notag	\sum_{n = 0}^{\infty} B(5n) q^{n} &= R(q)(1 - 25q \frac{f^{6}_{5}}{f^{6}_{1}})\\
	& =\sum_{n=0}^{\infty}d(n)q^{n} . \sum_{n=0}^{\infty}C(5n)q^{n}
	\end{align}
	The last equality holds because of \eqref{eq1}. 	Now, equating the constant term on both sides, we get 
	$$ B(0)=d(0)C(0).$$
		Since, $d(0)>0,$ and $C(0)=1$, we obtain $B(0)>0.$
        	Comparing the coefficients of $q$ from both sides, we get 
		\begin{align*}
	B(5)= d(0)C(5)+d(1)C(0).
		\end{align*}
		We complete the proof because of $d(0)>0, C(5)<0, d(1)<0, C(0)=1$.
\end{proof}

\begin{proof}[Proof of  Theorem \ref{gj3}]
	For proving the  Theorem \ref{gj3},  we will use equation\cite[Equation (79)]{baruah-sarma} 
	\begin{align*} 
	\sum_{n = 0}^{\infty} D(n) q^{n} = \frac{f^{6}_{25}}{f^{6}_{5}\cdot R^{4}(q^5)} &\biggl(1 + 3qR(q^5) + 4q^2 R^{2}(q^5) + 2q^{3}R^{3}(q^{5}) + q^{4}R^{4}(q^5)\biggl)^{2}\\
	&\times \biggl(\frac{1}{R(q^5)} - q - q^{2}R(q^5)\biggl).
	\end{align*}
	Now by extracting the terms of the form $q^{5n + 1}$, and then dividing the expression by $q$, followed by replacing $q^5$ with, $q$ we arrive at the following equation:
	\begin{align}\label{eq8}
\notag 	\sum_{n = 0}^{\infty} D(5n + 1) q^{n} & = \frac{f^{6}_{5}}{f^{6}_{1}} R(q) \biggl( \frac{5}{R^{5}(q)} - 40q\biggl).\\
		& =\sum_{n=0}^{\infty}d(n)q^{n} . \sum_{n=0}^{\infty}A(5n+1)q^{n}
	\end{align}
 	Now, equating the constant term on both sides, we get 
 $$ D(1)=d(0)A(1).$$
 Since, $d(0)>0,$ and $A(1)>0$, we obtain $D(1)>0.$

\end{proof}

\section{Acknowledgement} 
The authors thank Prof. Nayandeep Deka Baruah and Abhishek Sarma for bringing to our attention a flaw in the initial draft of this paper. Suparno Ghoshal was supported by the Deutsche Forschungsgemeinschaft (DFG, German Research Foundation) under Germany's Excellence Strategy – EXC 2092 CASA – 390781972.


\begin{thebibliography}{99}
		\bibitem{andrews1981ramunujan} Andrews, G.E.: Ramunujan's ``Lost'' Notebook III. The Rogers-Ramanujan continued fraction. Adv. Math. \textbf{41}, 186--208 (1981)

\bibitem{baruah-sarma}	Baruah, N.D., Sarma, A. Sign patterns and congruences of certain infinite products involving the Rogers-Ramanujan continued fraction. Ramanujan J. {\bf 67}, 7 (2025). https://doi.org/10.1007/s11139-025-01050-5
	
	\bibitem{bcb3} Berndt, B.C.: Ramanujan's  Notebooks. Part III. Springer, New York (1991)
	\bibitem{Spirit} Berndt, B.C.: Number Theory in the Spirit of Ramanujan. Volume~34 of Student Mathematical Library. American Mathematical Society, Providence, RI (2006)
	\bibitem{cooper-qtpi} Cooper, S.: The quintuple product identity, Int. J. Number Theory \textbf{2}, 115--161 (2006)
\bibitem{gordon} Gordon, B.: A combinatorial generalization of the Rogers-Ramanujan identities. American J. Math. \textbf{83}, 393--399 (1961)

	\bibitem{gugg-rama} Gugg, C.: A new proof of Ramanujan's modular equation relating $R(q)$ with $R(q^5)$. Ramanujan J. \textbf{20} 163--177 (2009)
		\bibitem{hirschhorn} Hirschhorn, M.D.: On the expansion of Ramanujan's continued fraction. Ramanujan J. \textbf{2}, 521--527 (1998)
	\bibitem{ramanathan-acta} Ramanathan, K.G.: On Ramanujan's continued fraction. Acta Arith. \textbf{43}, 209--226 (1984)
	
	
	\bibitem{nb} Ramanujan, S.: Notebooks, 2 volumes. Tata Institute of Fundamental Research, Bombay (1957)
	\bibitem{lnb} Ramanujan, S.: The Lost Notebook and Other Unpublished Papers, Narosa, New Delhi (1988)
	
	\bibitem{richmond1978taylor} Richmond, B., Szekeres, G.: The Taylor coefficients of certain infinite products. Acta Sci. Math.(Szeged) \textbf{40}, 347--369 (1978)
	
	\bibitem{rogers1894}Rogers, L.J.: Second memoir on the expansion of certain infinite products. Proc. London Math. Soc. \textbf{25}, 318--343 (1894)
	
	\bibitem{rogers1921}Rogers, L.J.: On a type of modular relation. Proc. London Math. Soc. (2) \textbf{19}, 387--397 (1921)
	

	\bibitem{watson} Watson, G.N.: Theorems stated by Ramanujan (VII ): Theorems on continued fractions. J. London Math. Soc. \textbf{4}, 39--48 (1929)
	

	
	\bibitem{yi}  Yi, J.: Modular equations for the Rogers-Ramanujan continued fraction and the Dedekind eta-function. Ramanujan J. \textbf{5}, 377--384 (2001)
\end{thebibliography}
\end{document}